\documentclass[a4paper,12pt]{amsart}
\usepackage[margin=3cm,head=3cm]{geometry}
\usepackage{standalone}
\usepackage{amsfonts}
\usepackage{amsthm}
\usepackage{amssymb}
\usepackage{amsmath}
\usepackage{theoremref}
\usepackage{enumerate}
\usepackage{bbm}
\usepackage{bm}
\usepackage{pgfplots}
\pgfplotsset{compat=1.18} 
\usepgfplotslibrary{fillbetween}
\usetikzlibrary{intersections}
\usepackage{mathtools}

\newcommand{\A}{\mathcal{A}}

\newcommand{\bT}{\mathbb{T}}
\newcommand{\bN}{\mathbb{N}}
\newcommand{\bD}{\mathbb{D}}
\newcommand{\bZ}{\mathbb{Z}}

\newcommand{\bC}{\mathbb{C}}
\newcommand{\ra}{\rightarrow}
\newcommand{\Ra}{\Rightarrow}

\newcommand{\D}{\mathbb{D}}

\newtheorem{thm}{Theorem}[section]
\newtheorem*{thm*}{Theorem}
\newtheorem{lemma}[thm]{Lemma}

\newtheorem{cor}[thm]{Corollary}

\theoremstyle{definition}
\newtheorem{example}[thm]{Example}
\theoremstyle{definition}
\newtheorem{remark}[thm]{Remark}

\begin{document}

\author{Linus Bergqvist}
\address{Department of Mathematics, Stockholm University, 106 91 Stockholm, Sweden}
\email{linus@math.su.se} 

\title{A comparison theorem for Clark-Aleksandrov measures on the torus}
\keywords{Clark Aleksandrov-measures, Polydiscs, de Branges-Rovnyak spaces}
\subjclass[2010]{28A25, 28A35 (primary); 32A10 (secondary)}

\maketitle

\begin{abstract}
    Given two Clark-Aleksandrov measures $\sigma^1$ and $\sigma^2$ on $\bT^2$, we prove a theorem relating the property that $\sigma^1 \ll \sigma^2$ to containment of a concrete function in a certain de Branges-Rovnyak space. We show that our theorem is applicable for all rational inner functions on $\bD^2$, and we provide several examples of how the theorem can be applied to investigate Clark-Aleksandrov measures related to such functions. 
\end{abstract}

\section{Background}

In Theorem (III-11) of \cite{Sarason}, Sarason proved a comparison theorem for Clark-Aleksandrov measures on $\bT$ that relates mutual absolute continuity and $L^2$ integrability to containment of a concrete function in a de Branges-Rovnyak space. Among other things, this result can be used to show that a Borel measure on $\bT$ has an atom at a point $z_0 \in \bT$ if and only if a related function is contained in a specific de Branges-Rovnyak space. 

More recently in \cite{ComparisonTheoremBall}, Aleksandrov and Doubtsov generalized this result to Clark-Aleksandrov measures on the unit ball of $\bC^n$ for $n \geq 2$ by using results about the singular parts of pluriharmonic measures on the unit ball.

In this paper, we give a generalization of Sarason's Theorem (III-11) to Clark-Aleksandrov measures on the torus $\bT^2$, and we give several examples of how this theorem can be applied. Before proceeding with preliminaries and notation, we give a short description of the proofs and structure of this paper. In the proof of Theorem (III-11) in \cite{Sarason}, it is important that the measure $\sigma^1$ (using the notation of Theorem \ref{MainTheorem} below) is a singular measure. But this might be a little misleading when trying to generalize the result, since the key property needed from the measure is not that it is singular with respect to Lebesgue measure, but rather that it is carried by a null set of the annihilators of the disc algebra. However, by the F. and M. Riesz theorem the measures on $\bT$ with this property are precisely the singular measures, so for measures on $\bT$, these properties are equivalent. This is no longer the case for measures on $\bT^n$, and finding a suitable substitute for the F. and M. Riesz theorem is the main difficulty when trying to generalize this result.

Combining Lemmas \ref{FirstImplicationGeneral} and \ref{RestrictNullset} below gives a generalization of Sarason's result for measures on $\bT^n$ that have the property that they are carried by a null set for the annihilators of the polydisc algebra. However, due to the product structure of $\bT^n$, pluriharmonic measures on $\bT^n$ may have degenerate parts, i.e. parts of the form $\sigma_k \times m_{n-k}$ where $\sigma_k$ is a pluriharmonic measure on $\bT^k$, where $k < n$, and $m_{n-k}$ denotes Lebesgue measure on $\bT^{n-k}$ (or measures obtained by permuting the variables of such a measure). Since the degenerate parts become notationally cumbersome for large $n$, we restrict our analysis of such measures to the case $n=2$, although it should be pointed out that similar results likely hold for $n > 2$ as well. We prove that our comparison theorem holds for degenerate pluriharmonic measures as well (see Lemmas \ref{FirstImplicationGeneral} and \ref{DegenerateCaseForMainTheorem}), and by combining all of these results, we obtain our main result, Theorem \ref{MainTheorem}, which is a generalization of Sarason's theorem to pluriharmonic measures on $\bT^2$.

As an application, we show that Theorem \ref{MainTheorem} is always applicable for Clark-Aleksandrov measures corresponding to rational inner functions on $\bD^2$, and we give a few examples of how the theorem can be applied to compare Clark-Aleksandrov measures on $\bT^2$ corresponding to such functions. 

Finally, there are indications that Clark-Aleksandrov measures corresponding to inner functions more generally will have non-degenerate part carried by some set with the desired property (see for example Theorem \ref{6.1.2. Rudin} below), but a proof of this statement would require a finer analysis of the non-closed null sets of annihilators of the polydisc algebra. This is illustrated in Example \ref{RudinNotEnough}.

\section{Preliminaries}

We will now give a very brief overview of the concepts and notation that will be used in this paper. For more on Clark-Aleksandrov measures on the $n$-torus we refer the reader to \cite{Doub}, in which these measures were originally introduced, as well as \cite{Clark1} and \cite{Clark2}, and for more on de Branges-Rovnyak spaces on $\bD$ we refer the reader to \cite{Sarason}. For more on polydisc analysis in general, see the textbook \cite{Rudin}.

Let $C(z,w)$, $z, w \in \bD^n$, denote the (product) \emph{Cauchy kernel}
$$
C(z,w) := \prod_{j=1}^n \frac{1}{1- z_j \overline{w_j}},
$$
for the polydisc $\bD^n := \{z = (z_1, \ldots, z_n): z_1, \ldots, z_n \in \bD \}$.
For a (complex) Borel measure $\mu$ on $\bT^n := \{z = (z_1, \ldots, z_n): z_1, \ldots, z_n \in \bT \}$, we denote by $\mu_+$ the \emph{Cauchy transform} of $\mu$, which is defined as
$$
\mu_+(z) := \int_{\bT^n} C(z,w) d \mu(w) = \sum_{k \in \bN^n} \hat{\mu}(k) z^k, \quad z \in \bD^n,
$$
where $\hat{\mu}(k) := \int_{\bT^n} \overline{w}^k d \mu(w)$, in standard multi-index notation. 

From the definition of $\mu_+$ it follows that the measures $\mu$ for which $\mu_+ \equiv 0$ are precisely the measures such that $\hat{\mu}(k) = 0$ for all $k \in \mathbb{N}^n$. By continuity of Borel measures as linear functionals on the continuous functions on $\bT^n$, these measures are precisely the measures that annihilate all functions of the form $\overline{f}$, where $f$ is a continuous function on $\overline{\bD^n}$, the closure of $\bD^n$, whose restriction to $\bD^n$ is holomorphic, i.e. where $f$ is a function in the \emph{polydisc algebra} $A(\bD^n)$. We denote by $A(\bD^n)^\perp$ the class of Borel measures on $\bT^n$ that annihilate $A(\bD^n)$. By the above discussion it follows that all measures in the kernel of the Cauchy transform are of the form $\overline{\mu}$ for some $\mu \in A(\bD^n)^\perp$.

Now, let $b: \bD^n \ra \bD$ be a non-constant holomorphic function. For $\alpha \in \bT$, the quotient
$$
\frac{1- |b(z)|^2}{|\alpha - b(z)|^2} = \text{Re} \left(\frac{\alpha + b(z)}{\alpha - b(z)} \right), \quad z \in \bD^n
$$
is positive and $n$-harmonic (i.e. it is harmonic in each variable separately), and thus there exists a (unique) positive Borel measure on $\bT^n$  whose Poisson integral
$$
P[\mu](z) := \int_{\bT^n} \prod_{j=1}^n \frac{1-|z_j|^2}{|\zeta_j - z_j|^2} d\mu(\zeta_1, \ldots, z_n), \quad z \in \bD^n,
$$
equals the quotient above. We call this measure the \emph{Clark-Aleksandrov measure} corresponding to the function $b$ and the parameter $\alpha$, and we denote this measure by $\sigma_\alpha[b]$. When it is contextually clear what the underlying function is, we will simply write $\sigma_\alpha$ instead of $\sigma_\alpha[b]$.

A positive measure on $\bT^n$ is called a \emph{pluriharmonic measure} if its Poisson integral is the real part of some holomorphic function on $\bD^n$. These measures are sometimes called RP-measures, and following the notation in \cite{Rudin}, we will denote the class of pluriharmonic measures on $\bT^n$ by RP$(\bT^n)$. By Theorem $2.4.1.$ in \cite{Rudin}, a pluriharmonic measures are characterized by a condition on their Fourier coefficients. Namely, a real measure $\mu$ is a pluriharmonic measure if and only if all of its mixed Fourier coefficients vanish. That is, $\mu$ is an pluriharmonic measure if and only if
\begin{equation} \label{RPmeasureVanishingCoeff}
\int_{\bT^n} z_1^{k_1} \cdots z_n^{k_n} d\mu(z) = 0
\end{equation}
unless all $k_j \in \bZ$ satisfy $k_j \geq 0$ for $j=1, \ldots, n$, or $k_j \leq 0$ for $j=1, \ldots n$.

By their defining property, all Clark-Aleksandrov measures are pluriharmonic measures. It may seem like Clark-Aleksandrov measures constitute an obscure class of measures, but in fact, they are no more obscure than general pluriharmonic measures. Every (positive) pluriharmonic measure of mass $1$ on $\bT^n$ is the Clark-Aleksandrov measure corresponding to $\alpha=1$ and some holomorphic function $\phi: \bD^n \ra \bD$. If the pluriharmonic measure is singular with respect to Lebesgue measure, then the corresponding function $\phi$ will be \emph{inner}, where an inner function means a holomorphic function $I:\bD^n \ra \bD$ that has radial limits of modulus $1$ Lebesgue almost everywhere on $\bT^n$. The proof of this correspondence between Clark-Aleksandrov measures and pluriharmonic measures when $n \geq 2$ is more or less analogous to the proof of the corresponding statement for $n=1$ (see Lemma $2.2.$ of \cite{Clark2} for details).

Direct inspection shows that
$$
k_w(z) = k_b(z, w) := (1-b(z) \overline{b(w)}) C(z,w), \quad z,w \in \bD^n
$$
is conjugate symmetric and positive definite, and thus it is 
the reproducing kernel of some reproducing kernel Hilbert space on $\D^n$. We call this Hilbert space the \emph{de Branges-Rovnyak space} corresponding to $b$, and we denote it by $\mathcal{H}(b)$. Note that if $I$ is inner, then $\mathcal{H}(I)$ is just the \emph{Model space} corresponding to $I$, i.e. $\mathcal{H}(I) = H^2(\bD^n) \ominus I H^2(\bD^n)$.

We can now define 
$$
(U_{b, \alpha} k_w)(\cdot) := (1- \alpha \overline{b(w)}) C( \cdot, w).
$$
As in one variable, this operator has a unique extension to a unitary operator from $\mathcal{H}(b)$ onto $H^2(\sigma_\alpha)$, where $H^2(\sigma_\alpha)$ denotes the closed linear span of the holomorphic polynomials in $L^2(\sigma_\alpha)$ (this is a special case of Theorem $3.1.$ in \cite{ComparisonTheoremBall}).

Furthermore, we can define $V_{b,\alpha}: L^2(\sigma_\alpha[b]) \ra \mathcal{H}(b)$ by
\begin{equation} \label{(3.1.) AlekDoub}
(V_{b,\alpha} g)(z) := (1 - \overline{\alpha} b(z)) (g \sigma_\alpha)_+ (z), \quad g \in L^2(\sigma_\alpha), z \in \bD^n.
\end{equation}

When the value of $\alpha$ is contextually clear, we will sometimes omit it and simply write $V_b$. 

We will use the following theorem from \cite{ComparisonTheoremBall}. 

\begin{thm}[Special case of Theorem $3.2.$ in \cite{ComparisonTheoremBall}] \label{Theorem 3.2. AlekDoub}

For each $\alpha \in \bT$, formula \eqref{(3.1.) AlekDoub} defines a partial isometry from $L^2(\sigma_\alpha)$ into $\mathcal{H}(b)$. The restriction of $V_{b, \alpha}$ to $H^2(\sigma_\alpha)$ coincides with $U_{b, \alpha}^*$, and in particular 
\begin{equation} \label{(3.2.) AlekDoub}
\left(V_{b, \alpha} C( \cdot, w) \right)(z) = \frac{1}{1- \alpha \overline{b(w)}} k_b(z,w).
\end{equation}

\end{thm}

\begin{remark}
Note that this implies that the restriction of $V_{b, \alpha}$ to $H^2(\sigma_\alpha)$ is a unitary operator onto $\mathcal{H}(b)$.
\end{remark}

Finally, a Borel set $S \subset \bT^n$ is called a \emph{carrier} of the measure $\mu$ if $\mu(A) = \mu(A \cap S)$ for all Borel subsets $A \subset \bT^n$, and $S$ is called a \emph{null set for} $A(\bD^n)^\perp$ if $|\mu|(S) = 0$ for all $\mu \in A(\bD^n)^\perp$. For a Borel set $S \subset \bT^n$, we denote by $\mu|_{S}$ the restriction of $\mu$ to $S$ defined by $\mu|_{S}(A) := \mu(A \cap S)$ for all Borel sets $A \subset \bT^n$.

\section{Degenerate pluriharmonic measures and decompositions of measures on $\bT^2$} \label{Section: DegenerateParts}

Before proving our main theorem, we must first address an issue arising as a consequence of the product structure of the $n$-torus -- namely that a pluriharmonic measure on $\bT^n = \bT^k \times \bT^l$, where $k + l = n$, may have a component carried by $E \times \bT^l$, where $E$ is a set of Lebesgue measure zero on $\bT^k$. In fact, for $\bT^2 = \bT \times \bT$, the product of \emph{any} singular measure on $\bT$ with Lebesgue measure on $\bT$ will be a pluriharmonic measure, and in general, the product of a Borel measure $\mu$ on $\bT^k$ with Lebesgue measure on $\bT^l$  will be a pluriharmonic measure if and only if $\mu \in$ RP$(\bT^k)$. What is less obvious however, is that the converse also holds: all pluriharmonic measures carried by $E \times \bT^l$ are of this form. This is a consequence of a theorem due to Ahern, which will be presented below.  

Denote by $\pi \mu$ the measure on $\bT^k$ defined by $(\pi \mu)(S) = \mu(S \times \bT^l)$ for every Borel set $S \subset \bT^l$, and let $m_s$ be Lebesgue measure on $\bT^s$ for $s \in \bN$. Then the following holds.
\begin{thm} [Proposition 1 in \cite{Ahern}]\label{Ahern}
    Let $\mu \in RP(\bT^n)$ and let $E \subset \bT^k$ be a Borel set such that $m_k(E) = 0$; then $\mu|_{E \times \bT^l} = (\pi \mu)|_E \times m_l$. 
\end{thm}

In general, a pluriharmonic measure will not be entirely degenerate, but may have \emph{parts} that are degenerate. In order to analyze the degenerate parts of pluriharmonic measures, we will need the following (unpublished) decomposition of annihilators of the bidisc algebra due to Brian Cole. 
\begin{thm}[Theorem $29.$ in \cite{Gamelin}] \label{Cole}
Let  $\nu \in A(\bD^2)^\perp$. Then
$$
\nu = \nu_0 + \nu_1 + \nu_ 2,
$$
where $\nu_0, \nu_1, \nu_2$ are mutually singular measures in $A(\bD^2)^\perp$; the measure $\nu_0$ is absolutely continuous with respect to some representing measure for the origin; $\nu_1 = h(\zeta_1, \zeta_2) d \tau(\zeta_1) dm(z_2)$, where $\tau$ is a measure on $\bT$ carried by a set of Lebesgue measure zero, $h \in L^1(\tau \times dm)$ and $h(\zeta_1, \cdot)$ belongs to the Hardy space $H^1_0(\bT)$ -- consisting of functions in $H^1$ that vanish at the origin -- for each fixed $\zeta_1 \in \bT$; and $\nu_2$ has a description similar to that of $\nu_1$ except that the roles of $\zeta_1$ and $\zeta_2$ are interchanged.   
\end{thm}

By applying a change of variables and the fact that pluriharmonic measures are characterized by satisfying \eqref{RPmeasureVanishingCoeff}, Cole's decomposition can be used to obtain a similar decomposition for measures in RP$(\bT^2)$. 

Now denote by $T$ the change of variables $T: (z_1, z_2) \mapsto (z_1 \overline{z_2})$, and denote by $T_*(\mu)$ the pushforward measure, defined by $T_*(\mu)(S) := \mu(T^{-1}(S))$ for all Borel sets $S$. If $\mu \in$ RP$(\bT^2)$, then $\nu := z_1 z_2 T_*(\mu) \in A(\bD^2)^\perp$, and conversely, for any real measure $\nu \in A(\bD^2)^\perp$, $\mu :=  T_*(\overline{z_1 z_2} \nu) \in$ RP$(\bT^2)$ (note that $T = T^{-1}$ is a bijection). By applying Cole's decomposition to $\nu$, we see that any measure $\mu \in$ RP$(\bT^2)$ can be decomposed as 
\begin{equation} \label{DecompRP}
\mu = \mu_0 + \mu_1 + \mu_2,  
\end{equation}
where $\mu_0, \mu_1, \mu_2$ are mutually singular measures in RP$(\bT^2)$; the measure $\mu_0 = T_*(\overline{z_1 z_2} \nu_0) = \overline{z_1} z_2 T_*(\nu_0)$ for some $\nu_0 \in A(\bD^2)^\perp$ that is absolutely continuous with respect to some representing measure for $(0,0)$; $\mu_1 = \tau_1 \times dm$, where $\tau_1$ is a real measure on $\bT$ carried by some set $E_1 \subset \bT$ of Lebesgue measure zero; $\mu_2 = dm \times \tau_2$ where $\tau_2$ is a real measure on $\bT$ carried by some set $E_2 \subset \bT$ of Lebesgue measure zero.

For more details on how $T$ relates measures in RP$(\bT^2)$ to measures in $A(\bD^2)^\perp$, see pages $4$ and $5$ in \cite{NecessaryConditionsOnRP} and pages $145$ and $146$ in \cite{Gamelin}.

\section{Main result}

We are now ready to prove our main result, which establishes a connection between absolute continuity of two pluriharmonic measures and containment of a concrete function in a certain de Branges-Rovnyak space. We begin by proving the first implication: namely that mutual absolute continuity implies containment of the function in a certain de Branges-Rovnyak space.

\begin{lemma} \label{FirstImplicationGeneral}
Let $b_1, b_2: \bD^n \ra \bD$ be non-constant holomorphic functions, and let $\sigma^1 = \sigma_\alpha[b_1]$ and $\sigma^2 = \sigma_\alpha[b_2]$, where $\alpha \in \bT$, be the corresponding Clark-Aleksandrov measures. As before, let $k_{b_1}( \cdot , w)$ denote the reproducing kernel of $\mathcal{H}(b_1)$.

Then the following implication holds:

If $\sigma^1 \ll \sigma^2 \text{ and } \frac{d \sigma^1}{d \sigma^2} \in L^2(\sigma^2)$, then $\frac{\alpha - b_2}{\alpha - b_1} k_{b_1}( \cdot , w) \in \mathcal{H}(b_2)$
    for all $w \in \bD^n$.

\end{lemma}

\begin{proof}

We can assume that $\alpha=1$ without loss of generality. 

By applying the formulas \eqref{(3.1.) AlekDoub} and \eqref{(3.2.) AlekDoub}, we get that
\begin{multline} \label{(4.1) AlekDoub}
(1 - b_2(z)) (C(\cdot, w) \sigma^1)_+ (z) = (1 - b_2(z))\frac{1-b_1(z)}{1-b_1(z)} (C(\cdot, w) \sigma^1)_+ (z) \\
= \frac{1-b_2(z)}{1-b_1(z)} (V_{b_1, 1} C( \cdot, w))(z) = \frac{1-b_2(z)}{1-b_1(z)} \frac{1}{1- \overline{b_1(w)}} k_{b_1}(z,w).
\end{multline}
Equation \eqref{(4.1) AlekDoub} will be used several times in this paper.

By the definition of $V_{b_2}$ we have
$$
V_{b_2} \left( \frac{d \sigma^1}{d \sigma^2} C( \cdot, w) \right) = (1- b_2)  \left( \frac{d \sigma^1}{d \sigma^2} C( \cdot, w) \sigma^2 \right)_+ = (1- b_2) ( C(\cdot, w) \sigma^1)_+.
$$
By Theorem \ref{Theorem 3.2. AlekDoub} $V_{b_2}$ maps $L^2(\sigma^2)$ into $\mathcal{H}(b_2)$, and thus combining the above expression with \eqref{(4.1) AlekDoub} finishes the proof. 

\end{proof}

We will now prove the reverse implication. For this we will use the following lemma, which establishes an important connection between containment of $\frac{\alpha - b_2}{\alpha - b_1} k_{b_1}( \cdot , w)$ in $\mathcal{H}(b_2)$, and existence of a combination of $\sigma^1$ and $\sigma^2$ that annihilates the polydisc algebra. This will be used to determine when $\sigma^1 \ll \sigma^2$. 

\begin{lemma} \label{Containment of function in de Branges Rovnyak condition}
    Let $w \in \bD^n$ be fixed, let $b_1, b_2: \bD^n \ra \bD$ be non-constant holomorphic functions, and let $\sigma^1 = \sigma_1[b_1]$ and $\sigma^2 = \sigma_1[b_2]$ be the corresponding Clark-Aleksandrov measures. As before, let $k_{b_1}( \cdot , w)$ denote the reproducing kernel of $\mathcal{H}(b_1)$. 

    Then $\frac{1 - b_2}{1 - b_1} k_{b_1}( \cdot , w) \in \mathcal{H}(b_2)$ if and only if $C(\cdot, w) \sigma^1 - q_w \sigma^2 \in A(\bD^n)^\perp$ for some function $q_w \in H^2(\sigma^2)$. 
\end{lemma}

\begin{proof}
Since $V_{b_2}: H^2(\sigma^2) \mapsto \mathcal{H}(b_2)$ is unitary (see the remark following Theorem \ref{Theorem 3.2. AlekDoub}), there is a function $q_w \in H^2(\sigma^2)$ such that
$$
(1- \overline{b_1(w)})^{-1} \frac{1- b_2(z)}{1- b_1(z)} k_{b_1}( \cdot , w) = (V_{b_2} q_w)(z) = (1- b_2(z)) (q_w \sigma^2)_+ (z),
$$
if and only if $\frac{1 - b_2}{1 - b_1} k_{b_1}( \cdot , w) \in \mathcal{H}(b_2)$. By combining this with Equation \eqref{(4.1) AlekDoub}, we see that this is equivalent to the statement that
\begin{equation}\label{productEqual0}
(1- b_2(z)) (C(\cdot, w) \sigma^1 - q_w \sigma^2)_+ (z) = 0, \quad z \in \bD^n.
\end{equation}
Since both $(1-b_2(z))$ and $(C(\cdot, w) \sigma^1 - q_w \sigma^2)_+ (z)$ are analytic functions on $\bD^n$ and $1-b_2 \not\equiv 0$ by assumption, \eqref{productEqual0} holds if and only if
\begin{equation} \label{CauchyTransVanish}
(C(\cdot, w) \sigma^1 - q_w \sigma^2)_+ (z) \equiv 0.    
\end{equation}
Finally, \eqref{CauchyTransVanish} holds if and only if $(C(\cdot, w) \sigma^1 - q_w \sigma^2) \in A(\bD^n)^\perp$, which finishes the proof. \qedhere

\end{proof}

We are now ready to prove the second implication for parts of pluriharmonic measures that are carried by some null set for $A(\bD^n)^\perp$.

\begin{lemma} \label{RestrictNullset}
Let $b_1, b_2: \bD^n \ra \bD$ be non-constant holomorphic functions, and let $\sigma^1 = \sigma_\alpha[b_1]$ and $\sigma^2 = \sigma_\alpha[b_2]$, where $\alpha \in \bT$, be the corresponding Clark-Aleksandrov measures. As before, let $k_{b_1}( \cdot , w)$ denote the reproducing kernel of $\mathcal{H}(b_1)$. 

Let $N$ be a null set of $A(\bD^n)^\perp$. Then the following implication holds:

If $\frac{\alpha - b_2}{\alpha - b_1} k_{b_1}( \cdot , w) \in \mathcal{H}(b_2)$ for some $w \in \bD^n$, then $\sigma^1|_N \ll \sigma^2 \text{ and } \frac{d \sigma^1|_N}{d \sigma^2} \in L^2(\sigma^2)$.

\end{lemma}

\begin{proof}
We can assume that $\alpha=1$ without loss of generality.

By Lemma \ref{Containment of function in de Branges Rovnyak condition}, there is a function $q_w \in H^2(\sigma^2)$ such that
$$
(1- b_2(z)) (C(\cdot, w) \sigma^1 - q_w \sigma^2)_+ (z) = 0, \quad z \in \bD^n.
$$
Since $b_2(z) \neq 1$ on $\bD^n$, this means that 
$$
(C(\cdot, w) \sigma^1 - q_w \sigma^2)_+ (z) \equiv 0 \Ra (C(\cdot, w) \sigma^1 - q_w \sigma^2) \in A(\bD^n)^\perp.
$$
By assumption $N$ is a null set for $A(\bD^n)^\perp$, so
$$
C(\cdot, w) \sigma^1|_N - q_w \sigma^2_{N} \equiv 0.
$$

Thus
$$
\sigma^1|_N = \frac{q_w}{C(\cdot, w)} \sigma^2|_{N} \ll \frac{q_w}{C(\cdot, w)} \sigma^2,
$$
and in particular $\sigma^1|_N \ll \sigma^2$. Since $q_w \in H^2(\sigma^2)$ and $C(\cdot, w)^{-1}$ is bounded, we have that $d \sigma^1|_N /d \sigma^2 \in L^2(\sigma^2)$, which finishes the proof. \qedhere

\end{proof}

Finally, we will prove the second implication for the degenerate parts of pluriharmonic measures on $\bT^2$. The proof relies on the decompositions from Section \ref{Section: DegenerateParts}, and thus we restrict our attention to pluriharmonic measures on $\bT^2$.

\begin{lemma} \label{DegenerateCaseForMainTheorem}
Let $w \in \bD^2$ be fixed, let $\sigma^1, \sigma^2 \in$ RP($\bT^2)$ and let $E \subset \bT$ be a set of Lebesgue measure zero.

Then if $(C(\cdot, w) \sigma^1 - q_w \sigma^2) = \nu \in A(\bD^n)^\perp$ for some $q_w \in H^2(\sigma^2)$, we have that $\sigma^1|_{E\times \bT} \ll \sigma^2|_{E\times \bT}$ and $d \sigma^2|_{E\times \bT} / d \sigma^1|_{E\times \bT} \in L^2(\sigma^2)$. The same conclusion holds if $E \times \bT$ is replaced by $\bT \times E$. \qedhere

\end{lemma}

\begin{proof}
To begin with, by Ahern's theorem (Theorem \ref{Ahern}), $\sigma^1|_{E \times \bT} = \tau_1 \times dm$ and $\sigma^2|_{E \times \bT} = \tau_2 \times dm$  for some singular measures $\tau_1$ and $\tau_2$ on $\bT$. We are done if we can show that $\tau_1 \ll \tau_2$ and $d \tau_2 / d \tau_2 \in L^2(\tau_2)$.

We begin by showing that $\tau_1 \ll \tau_2$. If this was not the case, there would, by definition, be some set $S \subset \bT$ such that $\tau_2(S)=0$, but $\tau_1(S) \neq 0$. By restricting to $S \times \bT$ and using Cole's decomposition (Theorem \ref{Cole}) on $\nu$, we see that  
$$
C(\cdot, w) \tau_1|_S \times dm = h(\cdot) \tau_3|_S \times dm,
$$
where $\tau_3$ is some singular measure on $\bT$ and $h(\zeta_1, \cdot)$ belongs to $H^1_0(\bT)$ for every fixed value of $\zeta_1 \in \bT$.

But this is impossible, since for each fixed $\zeta_1 \in \bT$, the function $h(\zeta_1, \zeta_2)(1-\zeta_2 \overline{w_2}) \in H^1_0(\bD)$ has to be constant as a function of $\zeta_2 \in \bT$, but it also has to vanish for $\zeta_2 = 0$. This can only happen if the function (and thus $\tau_1|_S$) is identically equal to $0$, which contradicts our assumption on $S$.

Now denote the Radon-Nikodym derivative $d \tau_2 / d \tau_1$ by $f$. It remains to show that $f \in L^2(\tau_2)$.

We have that
$$
(C((\zeta_1 , \zeta_2), w) f(\zeta_1) - q_w(\zeta_1, \zeta_2)) \tau_1 \times dm \in A(\bD^2)^\perp,
$$
so 
\begin{equation} \label{degenerateDerivativeInL2}
C((\zeta_1, 0), w) f(\zeta_1) - q_w(\zeta_1, 0)) = 0. 
\end{equation}
Note that $q_w(\zeta_1, 0)$ is well-defined, at least as a formal power series, since $q_w \in H^2(\sigma^2)$. Furthermore
$$
\int_{\bT} |q_w(\zeta_1, 0)|^2 d \tau_2(\zeta_1) = \int_{\bT} \left| \int_{\bT} q_w(\zeta_1, \zeta_2) dm(\zeta_2) \right|^2 d \tau_2(\zeta_1) \leq \sigma^2(\bT^2) \| q_w \|^2_{L^2(\sigma^2)} < \infty,
$$
and thus \eqref{degenerateDerivativeInL2} implies that $f \in L^2(\tau_2)$, as was to be shown. 

The proof for $\bT \times E$ is done in the same way. \qedhere

\end{proof}

By combining the above Lemmas, we obtain the main theorem of this paper.

\begin{thm} \label{MainTheorem}
Let $b_1, b_2: \bD^2 \ra \bD$ be non-constant holomorphic functions, and let $\sigma^1 = \sigma_\alpha[b_1]$ and $\sigma^2 = \sigma_\alpha[b_2]$, where $\alpha \in \bT$, be the corresponding Clark-Aleksandrov measures. As before, let $k_{b_1}( \cdot , w)$ denote the reproducing kernel of $\mathcal{H}(b_1)$. 

Let $\sigma^1 = \sigma^1_0 + \sigma^1_1 + \sigma^1_2$ be the decomposition of $\sigma_1$ as in \eqref{DecompRP}. Then if there is a set $N \subset \bT^2$ which is a null set of $A(\bD^2)^\perp$ that carries $\sigma^1_0$, then the following conditions are equivalent:

\begin{enumerate}
    \item $\sigma^1 \ll \sigma^2 \text{ and } \frac{d \sigma^1}{d \sigma^2} \in L^2(\sigma^2)$,
    \item $\frac{\alpha - b_2}{\alpha - b_1} k_{b_1}( \cdot , w) \in \mathcal{H}(b_2)$
    for all $w \in \bD^2$,
    \item $\frac{\alpha - b_2}{\alpha - b_1} k_{b_1}( \cdot , w) \in \mathcal{H}(b_2)$
    for some $w \in \bD^2$. 
    
\end{enumerate}

\end{thm}

\begin{proof}

$(i) \Ra (ii)$ is just Lemma \ref{FirstImplicationGeneral} for $n=2$.

$(ii) \Ra (iii)$ is trivial.

$(iii) \Ra (i)$. We are done if we can show that $\sigma^1_j \ll \sigma^2$ and $d \sigma^1_j / d \sigma^2 \in L^2(\sigma^2)$ for $j=0,1,2$. 

Applying Lemma \ref{Containment of function in de Branges Rovnyak condition} and Lemma \ref{DegenerateCaseForMainTheorem} gives the statement for $j=1$ and  $2$.

Since $N$ is a carrier for $\sigma^1_0$ by assumption, we have that $\sigma^1_0 = \sigma^1_0|_N$, and thus Lemma \ref{RestrictNullset} proves the statement for $j=0$, which finishes the proof.

\end{proof}

\begin{remark}
    If \emph{every} real representing measures for $(0,0)$ is carried by some set $S$ such that $\{(z_1, \overline{z_2}): (z_1, z_2) \in S \}$ is a null set for $A(\bD^2)^\perp$, then Theorem \ref{MainTheorem} will be applicable for every non-constant function $b_1: \bD^2 \ra \bD$. 
\end{remark}

\section{Applications}

In order to apply Theorem \ref{MainTheorem}, we must determine what functions $b_1$ have Clark-Aleksandrov measures (corresponding to some parameter $\alpha$) whose non-degenerate part is carried by some null set of $A(\bD^n)^\perp$. The following theorem relates compact null sets to zero sets of the polydisc algebra, which will be used to conclude that Clark-Aleksandrov measures corresponding to a large class of inner functions are indeed carried by such sets. Here, a compact subset $K \subset \bT^n$ is called a \emph{zero set of $A(\bD^n)$} if there is some $f \in A(\bD^n)$ such that $f \equiv 0$ on $K$, and $f \neq 0$ on $\overline{\bD^n} \setminus K$. 

\begin{thm}[Part of Theorem $6.1.2.$ in \cite{Rudin}] \label{NullsetTheorem6.1.2.} \label{6.1.2. Rudin}
Let $K$ be a compact subset of $\bT^n$. Then $K$ is a zero set of $A(\bD^n)$ if and only if $|\mu|(K) = 0$ for every $\mu \in A(\bD^n)^\perp$. 
\end{thm}

An interesting consequence of this theorem is that every compact subset of a zero set of $A(\bD^n)$ is also a zero set of $A(\bD^n)$.

An obvious candidate for the carrier $S_{\sigma^1}$ is the entire support of the measure $\sigma^1$. By Theorem \ref{6.1.2. Rudin}, the support satisfies the assumptions on $S_{\sigma^1}$ if and only if it is a level set for some function in $A(\bD^n)$.

If $I$ is an inner function, then the corresponding Clark-Aleksandrov measure $\sigma_\alpha[I]$ has support contained in the set 
$$
\overline{\{ \zeta \in \bT^n: \lim_{r \ra 1^-} I(r \zeta) = \alpha \}}
$$
(see Lemma $2.1.$ in \cite{Clark2} or Proposition $2.1.$ in \cite{Nell}). Thus the assumption of Theorem \ref{MainTheorem} is satisfied whenever $b_1$ is an inner function in $A(\bD^2)$ (i.e. when $b_1$ is a rational inner function with no singularities on the boundary, see Theorem $5.2.5.$ $(b)$ in \cite{Rudin}). In fact, the assumption  will always be satisfied when $b_1$ is a rational inner function on $\bD^2$.

By Theorem $5.2.5.$ of \cite{Rudin}, every rational inner function is of the form 
$$
e^{iv} z^d \frac{\tilde{p}(z)}{p(z)}, \quad z \in \bD^n,
$$
where $v \in [0, 2 \pi)$ $,d \in \bN^n$, $p(z)$ is a polynomial of polydegree $(d_1, \ldots, d_n)$ with no zeros in $\bD^n$, and $\tilde{p}$ is the reflection of $p$, defined by
\begin{equation} \label{p tilde}
\tilde{p}(z) := z_1^{d_1} \cdots z_n^{d_n} \overline{p \left(\frac{1}{\overline{z_1}}, \ldots , \frac{1}{\overline{z_n}} \right)}.
\end{equation}
Note that, unlike finite Blaschke products, rational inner functions in several variables may have singularities on the boundary (see Example \ref{Ex: anti-diagonal} below for an example of such a function). 

For most purposes it suffices to analyze rational inner functions of the form $\tilde{p}/p$. 
If $b_1$ is a rational inner function of the form $\tilde{p}/p$, then
\begin{equation} \label{ClarkSupportRIF}
\overline{\{ \zeta \in \bT^n: \lim_{r \ra 1^-} b_1(r \zeta) = \alpha \}} = \{ \zeta \in \bT^n: \tilde{p}(\zeta) - \alpha p(\zeta) = 0 \},
\end{equation}
(see Theorem $2.6.$ of \cite{AJM}). If $\tilde{p} - \alpha p$ does not vanish inside $\bD^n$, then we can apply Theorem \ref{NullsetTheorem6.1.2.} to conclude that the corresponding Clark measure is carried by a null set for $\A(\bD^n)^\perp$, and thus Lemma \ref{RestrictNullset} is applicable.

Furthermore if $\phi = \tilde{p}/p$ is a RIF on $\bD^2$, then $\tilde{p}(z) - \alpha p(z)$ can only vanish at a point $(z'_1, z'_2) \in \bD^2$ if either $(z_1 - z'_1)$ or $(z_2 - z'_2)$ is a factor of $\tilde{p}(z) - \alpha p(z)$ (see the proof of Lemma $4.1.$ in \cite{Clark2} for details). Clearly there can only be finitely many such factors. Thus the set \eqref{ClarkSupportRIF} can be written as a union of finitely many coordinate parallel lines (corresponding to the degenerate part of the Clark measure), and one component that is the zero set of some polynomial that does not vanish inside $\bD^2$ (something similar likely holds for RIFs on $\bD^n$ with $n > 2$ as well). 

Thus the assumptions of Theorem \ref{MainTheorem} are satisfied for all RIFs $\phi = \tilde{p}/p$ on $\D^2$ (even when $p$ vanishes on the boundary). We formulate this as a theorem.

\begin{thm} \label{MainResultRIFsD^2}
Theorem \ref{MainTheorem} is applicable whenever $b_1 \in H^2(\bD^2)$ is a rational inner function of the form $b_1 = \tilde{p}/p$, where $p(z) \neq 0$ in $\bD^2$. 
\end{thm}

\begin{example} \label{Ex: anti-diagonal}
Let $b_2 = z_1 z_2$. By Example $4.2.$ in \cite{Doub}, $\sigma_\alpha[b_2]$ is Lebesgue measure on the line $\{(e^{iv}, \alpha e^{-iv}) \in \bT^2: v \in [0, 2 \pi) \}$. Now let $\sigma^2 = \sigma_1[b_2]$, and let $b_1: \bD^2 \ra \bD$ with $b_1(0,0)=0$ be another holomorphic function with corresponding Clark-Aleksandrov measure $\sigma^1 := \sigma_1[b_1]$. Then by applying Theorem \ref{MainTheorem} for $w=0$, we see that $\sigma^1 \ll \sigma^2$ and $d \sigma^1 / d \sigma^2 \in L^2(d \sigma^2)$ (i.e. $\sigma^1$ has support on the anti-diagonal of $\bT^2$ satisfying the $L^2$ requirement) if and only if

\begin{equation} \label{Containment H(b) example 1}
\frac{1 - z_1 z_2}{1 - b_1(z_1, z_2)} \in \mathcal{H}(z_1 z_2).
\end{equation}

For example, if $b_1$ is the rational inner function
$$
b_1(z_1, z_2) := \phi(z_1, z_2) = \frac{2 z_1 z_2 - z_1 - z_2}{2 - z_1 - z_2},
$$
then
\begin{equation*}
\frac{1 - z_1 z_2}{1 - \phi(z_1, z_2)} = \frac{2 - z_1 - z_2 }{2}.
\end{equation*}
Since $\mathcal{H}(z_1 z_2) = H^2 \ominus z_1 z_2 H^2$ consists of all functions of the form $c + f(z_1) + g(z_2),$ where $f,g \in H^2(\bD)$ and $f(0) = g(0) = 0$, the containment \eqref{Containment H(b) example 1} is satisfied.

It may be interesting to note that another application of Theorem \ref{MainTheorem} shows that
\begin{multline*}
\frac{1 - z_1 z_2}{1 - \phi(z_1, z_2)} \frac{(1-\phi(z) \overline{\phi(w)})}{(1- z_1 \overline{w_1})(1- z_2 \overline{w_2})} = \frac{2 - z_1 - z_2 }{2}  \frac{(1- \frac{2 z_1 z_2 - z_1 - z_2}{2 - z_1 - z_2} \overline{\phi(w)})}{(1- z_1 \overline{w_1})(1- z_2 \overline{w_2})} \\
\frac{((2-z_1-z_2) - (2 z_1 z_2 - z_1 - z_2) \overline{\phi(w)})}{2 (1- z_1 \overline{w_1})(1- z_2 \overline{w_2})} \perp z_1 z_2 H^2,
\end{multline*}
for \emph{all} $w \in \bD^2$, which is not as obvious.

That $\sigma^1 \ll \sigma^2$ and $d \sigma^1 / d \sigma^2 \in L^2(\sigma^2)$ can also be verified in more direct ways. As mentioned earlier, the measure $\sigma^1$ is supported on the set
$$
\{ (z_1, z_2) \in \bT^2: (2 z_1 z-2 - z_1 - z_2) = 2 - z_1 - z_2 \iff z_1 z_2 = 1 \},
$$
and one can in fact calculate the density with respect to $\sigma^2$ directly. In Example $5.1.$ in \cite{Clark1}, it is shown that $d \sigma^1(\zeta) = |1-\zeta|^2 d m(\zeta)$, where $d m (\zeta)$ denotes normalized Lebesgue measure on the line $\{ (\zeta, \overline{\zeta}) \in \bT^2: \zeta \in \bT \}$ (i.e. $d m(\zeta) = d \sigma^2$). Thus $d \sigma^1 / d \sigma^2 = |1-\zeta|^2$, which is contained in $L^2(d \sigma^2)$ as expected.

However, $d \sigma^2 / d \sigma^1 = |1-\zeta|^{-2}$ does \emph{not} lie in $L^2(d \sigma^1)$, so the corresponding containment from Theorem \ref{MainTheorem} does not hold. That is
$$
\frac{2}{2-z_1 - z_2} \not\in H^2 \ominus \phi H^2.
$$
This is certainly true, since $2/(2-z_1-z_2)$ is not even contained in $H^2(\D^2)$. 

We can also go back and try to determine for what holomorphic functions $b_1: \bD^2 \mapsto \bD$ the containment \eqref{Containment H(b) example 1} does hold. As mentioned before, $\mathcal{H}(z_1 z_2)$ consists of all holomorphic functions of the form $c + f(z_1) + g(z_2)$ where $f,g \in H^2(\bD)$ and $f(0) = g(0) = 0$. Thus, we are looking for holomorphic functions on $\bD^2$ that map into $\bD$ that are of the form
\begin{equation} \label{first b_1}
b_1(z_1, z_2) = 1- \frac{1 - z_1 z_2}{c + f(z_1) + g(z_2)} = \frac{(c-1) + f(z_1) + g(z_2)  + z_1 z_2}{c + f(z_1) + g(z_2)}.
\end{equation}
Note that $c \neq 0$ since we want $b_1$ to be holomorphic in $\bD^2$.

If we limit our attention to finding rational inner functions $b_1 = \tilde{p}/p$, where $p$ is a polynomial of bidegree $(m,n)$ that does not vanish in $\bD^2$, and $\tilde{p}$ is given by \eqref{p tilde}, then we can explicitly describe all possible functions $b_1$. 

The numerator on the right hand side of \eqref{first b_1} contains no mixed term of bidegree higher than $(1,1)$, and since $c \neq 0$, this means that $m,n \leq 1$. In fact, $m=n=1$, since if $m$ or $n$ is zero, then $\tilde{p}$ will not contain the term $z_1 z_2$. 

This means that $c+ f(z_1) + g(z_2) = c + a z_1 + b z_2$, so 
$$
\tilde{p} = c z_1 z_2 + a z_2 + b z_1 =  (c-1) + a z_1 + b z_2 + z_1 z_2, 
$$
so $c= 1$ and $a=b$. That is, the possible functions are exactly the functions of the form $b_1(z) = (z_1 z_2 + a z_1 + a z_2)/(1 + a z_1 + a z_2)$.
Finally, since $b_1$ is holomorphic on $\bD^2$, the denominator cannot vanish there, and so $-1/a$ is not contained in the open disc with radius $2$. By setting $d=-1/a$, we see that $b_1$ is of the form
\begin{equation} \label{favRIFs}
b_1(z) = \frac{d z_1 z_2 - z_1 - z_2}{d - z_1 - z_2},
\end{equation}
where $|d| > 2$. Note that if we let $d \ra \infty$, we obtain the function $b_2 = z_1 z_2$. 

As we saw earlier
$$
\frac{1 - b_1(z_1, z_2)}{1- z_1 z_2} \not\in H^2 \ominus b_1 H^2
$$
for $b_1$ as in \eqref{favRIFs} with $d=2$, since in this case $(1- b_1(z_1, z_2))(1- z_1 z_2)$ does not even lie in $H^2$. However, for all $d$ such that $b_1$ from \eqref{favRIFs} belongs to $H^2$, the orthogonality requirement also holds. To see this, note that
$$
\frac{1- b_1(z_1, z_2)}{1- z_1 z_2} = \frac{d}{d - z_1 - z_2},
$$
so for all $f \in H^2$
$$
\int_{\bT^2} \frac{1}{\overline{d - z_1 - z_2}} \frac{d z_1 z_2 - z_1 - z_2}{d - z_1 - z_2} f(z) |dz| = \int_{\bT^2}  \frac{(z_1 z_2)}{\overline{d - z_1 - z_2}} \frac{\overline{d  - z_2 - z_1}}{d - z_1 - z_2} f(z) |dz| = 0,
$$
where the last equality holds since $(z_1 z_2 f(z))/(d - z_1 - z_2)$ vanishes at the origin. 

By applying Theorem $3.8$ in \cite{Clark2} we can describe the Clark measures corresponding to functions of the form \eqref{favRIFs} explicitly. Namely, for $|d| \geq 2$ the Clark measure corresponding to the function $b_1(z) := (d z_1 z_2 - z_1 - z_2)/(d - z_1 - z_2)$ and the parameter $\alpha = 1$, which we denote by $\sigma_d$, is given by
$$
d \sigma_d(\zeta) = \frac{1}{\left| \frac{\partial b_1}{ \partial z_2} (\zeta, \overline{\zeta}) \right|} d m(\zeta).
$$
Direct calculation now shows that $d \sigma_d(\zeta) = |d - 2 \text{Re}(\zeta)| d m(\zeta)$.

It may be of interest to note that \emph{all} of these measures can be described as linear combinations of the two measures corresponding to the “endpoint” functions $(2 z_1 z_2 - z_1 - z_2)/(2- z_1 - z_2)$ and $z_1 z_2$. 

\end{example}

The calculations showing that the orthogonality requirement automatically holds whenever $b_1$ from \eqref{favRIFs} belongs to $H^2$ seem to be an example of a more general phenomenon that appears with RIFs. More concretely, if $c \tilde{p_1} - \alpha p_1 =  \tilde{p_2} - \alpha p_2$ (which in particular means that the corresponding Clark measures are supported on the same curve), then $(\alpha - \phi_2)/(\alpha - \phi_1) = c p_1/p_2$, and as a consequence of the fact that the reflection $\tilde{p}$ satisfies $\tilde{p} = z^n \overline{p}$ on $\bT^n$, where $n \in \bN^n$ denotes the polydegree of $p$, we have that $(\alpha - \phi_2)(\alpha - \phi_1)^{-1} k_{\phi_1}(\cdot, 0)$ automatically lies in the model space $H^2 \ominus \phi_2 H^2$ whenever $p_1 / p_2$ lies in $H^2$. We clarify this in the following corollary.

\begin{cor}
Let $\phi_1 = \tilde{p}_1 / p_1$ and $\phi_2 = \tilde{p}_2 / p_2$  be two non-constant rational inner functions on $\bD^n$, let $n_j \in \bN^n$ denote the polydegree of $p_j$ for $j=1,2$ and let $\sigma^1 := \sigma_{\alpha}[\phi_1]$ and $\sigma^2 := \sigma_{\alpha}[\phi_2]$. Assume that $n_2 \geq n_1$, and that 
$$
\frac{\alpha - \phi_2}{\alpha - \phi_1} = c \frac{p_1}{p_2}
$$
for some constant $c \neq 0$.

Then $p_1/p_2 \in H^2$, if and only if  $\sigma^1 \ll \sigma^2$ and $d \sigma^1 / d \sigma^2 \in L^2(d \sigma^2)$. 
\end{cor}

Note that if $c=0$ then $\phi_2$ is constant.

\begin{proof}
 If $p_1 / p_2 \in H^2$, then for all $f \in H^2$
\begin{multline*}
\int_{\bT^n} \frac{\overline{\alpha - \phi_2(z)}}{\overline{\alpha - \phi_1(z)}} (1- \phi_1(0) \overline{\phi_1(z)})) \phi_2(z) f(z) |dz| \\
= \int_{\bT^n} (1- \phi_1(0) \overline{\phi_1(z)}) \frac{\overline{p_1(z)}}{\overline{p_2(z)}} \frac{\tilde{p}_2(z)}{p_2(z)} f(z) |dz| \\
= \int_{\bT^n} (1- \phi_1(0) \overline{\phi_1(z)})) \frac{\overline{p_1}(z)}{\overline{p_2}(z)} \frac{z^{n_2} \overline{p}_2(z)}{p_2(z)} f(z) |dz| \\
= \int_{\bT^n} (1- \phi_1(0) \overline{\phi_1(z)})) z^{n_2 - n_1} \frac{\tilde{p}_1(z)}{p_2(z)} f(z) |dz|
\\
= \int_{\bT^n} (1- \phi_1(0) \overline{\phi_1(z)})) \frac{p_1}{p_2} z^{n_2 - n_1} \phi_1(z) f(z) |dz| \\
= \int_{\bT^n} (\phi_1(z) - \phi_1(0))) \frac{p_1(z)}{p_2(z)} z^{n_2 - n_1} f(z) |dz| = 0,
\end{multline*}
where the last equality holds since the integrand lies in $H^2$ and vanishes at the origin. Applying Theorem \ref{MainTheorem} for $w=0$ finishes the proof of the first implication.

If instead $\sigma^1 \ll \sigma^2$ and $d \sigma^1 / d \sigma^2 \in L^2(d \sigma^2)$, then applying $\ref{MainTheorem}$ for $w=0$ shows that $(1- \phi_1(z) \overline{\phi_1(0)}) p_1 / p_2$ lies in $H^2 \ominus \phi_2 H^2$, which in particular means that it lies in $H^2$. Since $\phi_1$ is inner, it follows that $p_1/p_2 \in H^2$, which finishes the proof. \qedhere

\end{proof}

\begin{example} \label{Ex: containment L^2 ideals}
Let
$$
\phi_1(z) = \frac{\tilde{p}_1}{p_1} :=  \frac{4 z_1^2 z_2 - z_1^2 - 3 z_1 z_2 - z_1 + z_2}{4 - z_2 - 3z_1 - z_1 z_2 + z_1^2}
$$
and
$$
\phi_2(z) = \frac{\tilde{p}_2}{p_2} := \frac{2 z_1^2 z_2 - z_1 - 1}{2 - z_1 z_2 - z_1^2 z_2},
$$
and let $\sigma^1 := \sigma_{-1}[\phi_1]$ and $\sigma^2 := \sigma_{-1}[\phi_2]$.

It is noted in Remark $5.3.$ of \cite{Clark1} that $\sigma^1$ and $\sigma^2$ are supported on the same union of curves -- namely the union of the anti-diagonal $\{(e^{iv}, e^{-iv}) \in \bT^2: v \in [0, 2 \pi) \}$ and the vertical line $\{(1, e^{i v}): v \in [0, 2 \pi) \}$ -- and the corresponding weights with respect to Lebesgue measure on that curve are calculated.

Direct calculation shows that
$$
\frac{1 + \phi_2}{1 + \phi_1} = \frac{1}{4} \frac{4 - z_2 - 3z_1 - z_1 z_2 + z_1^2}{2 - z_1 z_2 - z_1^2 z_2} = \frac{1}{4} \frac{p_1}{p_2}.
$$
Since  both $p_1$ and $p_2$ have bidegree $(2,1)$, the previous corollary shows that $\sigma^1 \ll \sigma^2$ and $d \sigma^1 / d \sigma^2 \in L^2(d \sigma^2)$ if and only if $p_1 / p_2 \in H^2(\bD^2)$.  

Thus, we are asking whether $p_1$ is contained in the ideal -- which we will denote by $\mathcal{I}_{L^2}(p_2)$ -- which consists of all polynomials $q$ such that $q / p_2 \in L^2(\bT^2)$. Ideals of this sort were extensively studied in \cite{Knese}.

The ideal $\mathcal{I}_{L^2}(p_2)$ contains the ideal $\langle p_2 , \tilde{p}_2\rangle$, since all polynomials $q \in \langle p_2 , \tilde{p}_2\rangle$ have the property that $q/p_2$ is bounded on $\bT^2$.

A Gröbner basis for $\langle p_2, \tilde{p_2} \rangle$ is given by 
\begin{align*}
v_1 &:= (z_1 - 1)^2 = -2 x p_2 - (1+x) \tilde{p}_2 ,\\
v_2 &:= (3 z_1 + 2 z_2 - 5) = -4(1+xy)p_2 - (3 + 2y(1+x))\tilde{p}_2,
\end{align*}
but
\begin{multline*}
p_1 := 4 - z_2 - 3z_1 - z_1 z_2 + z_1^2 = \frac{5}{2} (z_1 - 1)^2 - \frac{z_1 + 1}{2} (3 z_1 + 2 z_2 - 5) + (z_1 - 1) \\
= \frac{5}{2} v_1 - \frac{z_1 + 1}{2} v_2 + (z_1 - 1), 
\end{multline*}
so $p_1$ is not in the ideal generated by $p_2$ and $\tilde{p}_2$ as it has remainder $z_1 - 1$ relative to the Gröbner basis $v_1, v_2$ (see \cite{Gröbner} for more on Gröbner bases).

However, from the above calculations it follows that $p_1 \in \mathcal{I}_{L^2}(p_2)$ if (and only if) $(1-z_1) \in \mathcal{I}_{L^2}(p_2)$. By using Cauchy's integral formula we see that
\begin{multline*}
\int_{\bT^2} \frac{|z_1 - 1|^2}{|2 - z_1 z_2 - z_1^2 z_2|^2} |dz| = \int_{\bT}|z_1 - 1|^2 \int_{\bT}\frac{1}{|2 - z_1 z_2 - z_1^2 z_2|^2} |dz_2| |d z_1| \\
= \frac{1}{2}\int_{\bT}|z_1 - 1|^2 \int_{\bT}\frac{1}{2 - z_1 z_2 - z_1^2 z_2} \frac{1}{1 - \overline{z_2\frac{(z_1 + z_1^2)}{2}}} |dz_2| |d z_1| \\
= \frac{1}{2}\int_{\bT}|z_1 - 1|^2 \int_{\bT}\frac{1}{(2 - z_1 z_2 - z_1^2 z_2)} \frac{1}{\left( 1 - \overline{z_2}\frac{\overline{z_1 + z_1^2}}{2} \right)} |dz_2| |d z_1| \\
= \frac{1}{2}\int_{\bT}|z_1 - 1|^2 \frac{1}{2 - \frac{|z_1 + z_1^2|^2}{2}}|d z_1| = \int_{\bT} \frac{|z_1 - 1|^2}{4 - |1 + z_1|^2}|d z_1| \\
= \int_{\bT} \frac{|z_1 - 1|^2}{(2 - |1 + z_1|)(2+|1+ z_1|)}|d z_1| < \infty,
\end{multline*}
where the last inequality follows since $|z_1-1|^2/(2-|1+z_1|) \leq 4$ for $z_1 \in \bT$. 
It follows that $p_1 \in \mathcal{I}_{L^2}(p_2)$, and thus $\sigma^1 \ll \sigma^2$ and $d \sigma^1 / d \sigma^2 \in L^2(d \sigma^2)$.

Conversely, $\sigma^2 \ll \sigma^1$ and $d \sigma^2 / d \sigma^1 \in L^2(d \sigma^1)$ if and only if $p_2/p_1 \in H^2$, i.e. if and only if $p_2 \in \mathcal{I}_{L^2}(p_1)$.

In Example $15.1$ of \cite{Knese}, Knese shows that the ideal $\mathcal{I}_{L^2}(p_1)$ is generated by the polynomials $v_1 = (1-z_2)^2, v_2 = (1-z_1)(1-z_2),$ and $v_3 = 1 - z_1 z_2$. A Gröbner basis for this ideal is given by $g_1 = (1-z_2)^2, g_2 = z_1+z_2-2$. That it is indeed a basis is verified by noting that $v_1 = g_1$, $v_3 = g_1 - z_2 g_2$, and $v_2 = -(g_2 + v_3) = - (g_2(1- z_2) + g_1)$.

It can now be checked directly that $p_2$ is \emph{not} contained in the ideal generated by these polynomials (its remainder with respect to $g_1, g_2$ is $r(z) := z_2 - 1 \neq 0$). 

Thus $p_2 / p_1 \not\in H^2$, and so it is not the case that $\sigma^2 \ll \sigma^1$ and $d \sigma^2/ \sigma^1 \in L^2(\sigma^1)$.

\end{example}

We end this note with an example where Theorem \ref{MainTheorem} \emph{is} applicable, i.e. where the non-degenerate part of $\sigma^1$ is carried by a null set of $A(\bD^2)^\perp$, but where one can not use Theorem \ref{6.1.2. Rudin} to deduce this. 

\begin{example} \label{RudinNotEnough}
For $\alpha \in \bT$, consider the Borel measures on $\bT^2$ given by $\sigma_\alpha := \sigma_\alpha[z_1 z_2]$. As we saw in Example \ref{Ex: anti-diagonal}, $\sigma_\alpha$ is normalized Lebesgue measure on the straight line $\{(e^{iv}, \alpha e^{-iv}) \in \bT^2: v \in [0, 2 \pi) \}$, and for every $\alpha$, this straight line will be a null set of $A(\bD^n)^\perp$ by Theorem \ref{6.1.2. Rudin}. Now let $\{ \alpha_j \}_{j=1}^\infty$ with $\alpha_j \in \bT$ be a dense sequence in $\bT$, and consider the measure
$$
\sigma := \sum_{j=1}^\infty \frac{1}{2^j} \sigma_{\alpha_j}. 
$$
This measure is a singular Borel probability measure on $\bT^2$, so as we noted in the introduction it will be the Clark-Aleksandrov measure corresponding to $\alpha = 1$ and some inner function $\phi: \bD^2 \ra \bD$ (again, see Lemma $2.2.$ of \cite{Clark2} for details). The set
$$
\bigcup_{j=1}^\infty \{(e^{iv}, \alpha_j e^{-iv}) \in \bT^2: v \in [0, 2 \pi) \} \subset \bT^2
$$
is a carrier of the measure $\sigma$, and furthermore it is a null set for $A(\bD^2)^\perp$ since countable unions of sets of measure zero have measure zero. It follows that Theorem \ref{MainTheorem} is applicable for $b_1 = \phi$ and $\sigma_1[b_1] = \sigma$. 

However, there is no closed null set for $A(\bD^2)^\perp$ that carries $\sigma$, since the support of $\sigma$ is all of $\bT^2$.

Thus it would be desirable to have a way of determining whether or not a given set is a null set for $A(\bD^2)^\perp$, which does not require that the set in question is closed.

\end{example}

In general, it would be interesting to know for which inner functions $b_1$ Theorem \ref{MainTheorem} is applicable. Will all Clark-Aleksandrov measures corresponding to inner functions (i.e. all singular pluriharmonic measures of mass $1$) have non-degenerate part carried by some null-set for $A(\bD^2)^\perp$, or is this only the case for certain inner functions?

\section*{Acknowledgements}
The author thanks Alan Sola for several valuable discussions, and Louis Hainaut for help and valuable discussions regarding the ideals in Example \ref{Ex: containment L^2 ideals}.


\begin{thebibliography}{99}

\bibitem{Ahern}
    Patrick R. Ahern
    \emph{: Inner functions in the polydisc and measures on the torus},
    Michigan Math. J. 20 (1973), 33--37.


\bibitem{ComparisonTheoremBall}
    Aleksei B. Aleksandrov and Evgueni Doubtsov 
    \emph{: Clark measures and de Branges-Rovnyak spaces in several variables}, Complex Var. Elliptic Equ., vol. 68, no. 2, (2023), 212-–221.

\bibitem{Clark2}
	J.T. Anderson, L. Bergqvist, K. Bickel, J.A. Cima, and A. A. Sola
	\emph{: Clark measures for rational inner functions II: general bidegrees and higher dimensions}, preprint available at
	https://arxiv.org/abs/2303.11248

 \bibitem{NecessaryConditionsOnRP}
    Linus Bergqvist
    \emph{: Necessary conditions on the support of RP-measures},
    preprint available at
    https://arxiv.org/abs/2304.03072v1

\bibitem{Clark1}
	K. Bickel, J.A. Cima, and A. A. Sola
	\emph{: Clark measures for rational inner functions},
	Michigan Math. J., to appear. https://doi.org/10.1307/mmj/20216046.

\bibitem{AJM}
        K. Bickel, J.E. Pascoe, and A.A. Sola 
        \emph{: Singularities of rational inner functions in higher dimensions},
        Amer. J. Math. 144, (2022), 1115–-1157.

\bibitem{Doub}
	Evgueni Doubtsov
	\emph{: Clark Measures on the Torus},
	Proc. Amer. Math. Soc. 148 (2020), 2009--2017.

\bibitem{Gröbner}
        Ralf Fröberg
        \emph{: An introduction to Gröbner bases},
        Pure and applied mathematics. Wiley, New York,
        (1998).

\bibitem{Gamelin}
	Theodore W. Gamelin
	\emph{: Uniform algebras on plane sets},
	in G. G. Lorentz (ed.), Approximation Theory, Academic Press, New York, 1973.

\bibitem{Nell}
    Nell Jacobsson
    \emph{: Clark measures on polydiscs associated to product functions and multiplicative embeddings},
    preprint available at
    https://arxiv.org/abs/2309.07150

\bibitem{Knese}
    Greg Knese
    \emph{: Integrability and regularity of rational functions},
    Proc. Lond. Math. Soc. \textbf{111} (6) (2015), 1261--1306.


\bibitem{Rudin}
	Walter Rudin
        \emph{: Function Theory in Polydisks},
	W. A. Benjamin, Inc.,
	New York-Amsterdam,
	(1969). 
	

\bibitem{Sarason}
	Donald Sarason 
	\emph{: Sub-Hardy Hilbert spaces in the unit disk},
	Lecture Notes in the Mathematical Sciences, vol. 10,
	Wiley, New York,
	(1994).

	

	


\end{thebibliography}
\end{document}